\documentclass[12pt]{article}

\usepackage{amsfonts}
\usepackage{amsmath}
\usepackage{amsthm}
\usepackage{graphicx}
\usepackage{xcolor, enumerate}

\newtheorem{theorem}{Theorem}
\newtheorem{lemma}[theorem]{Lemma}
\newtheorem{cor}[theorem]{Corollary}
\newtheorem*{defn}{Definition}
\newtheorem{q}{Question}

\newcommand\np{\vspace{1mm}}

\def\F{\mathcal{F}}

\def\A{\mathcal{A}}
\def\B{\mathcal{B}}
\def\P{\mu }

\def\a{\mathbf{a}}
\def\b{\mathbf{b}}

\DeclareMathOperator{\pos}{pos}
\DeclareMathOperator{\inv}{inv}
\DeclareMathOperator{\disp}{disp}

\def\power{\mathcal{P}(X)}

\begin{document}

\title{Correlation for permutations} 
\author{J. Robert Johnson\thanks{School of Mathematical Sciences, Queen Mary University of London, London E1 4NS, UK. 
E-mail: r.johnson@qmul.ac.uk.}
\and Imre Leader\thanks{DPMMS, University of Cambridge, Wilberforce Road, Cambridge CB3 0WB, UK. E-mail: i.leader@dpmms.cam.ac.uk.} \and Eoin Long\thanks{School of Mathematics, University of Birmingham, Edgbaston, Birmingham
B15 2TT, UK. E-mail: e.long@bham.ac.uk.}}
\maketitle

\begin{abstract}
In this note we investigate correlation inequalities for `up-sets' of permutations, in the spirit of the Harris--Kleitman inequality. We focus on two well-studied partial orders on $S_n$, giving rise to differing notions of up-sets. Our first result shows that, under the strong Bruhat order on $S_n$, up-sets are positively correlated (in the Harris--Kleitman sense). Thus, for example, for a (uniformly) random permutation $\pi$, the event that no point is displaced by more than a fixed distance $d$ and the event that $\pi$ is the product of at most $k$ adjacent transpositions are positively correlated. 
In contrast, under the weak Bruhat order we show that this completely fails: surprisingly, there are two up-sets each of measure $1/2$ whose intersection has arbitrarily small measure.

We also prove analogous correlation results for a class of non-uniform measures, which includes the Mallows measures. Some applications and open problems are discussed. 
\end{abstract}

\section{Introduction}

Let $X=\{1,2,\dots,n\}=[n]$. A family $\F \subset {\cal P}(X) = \{A: A \subset X\}$ is an \emph{up-set} if given $F \in \F$ and $F \subset G \subset X$ then $G \in \F$. The well-known and very useful Harris--Kleitman inequality \cite{harris,kleitman} guarantees that any two up-sets from ${\cal P}(X)$ are positively correlated. In other words, if $\A,\B\subset \power$ are both up-sets then
\[
\frac{|\A\cap\B|}{2^n}\geq \frac{|\A|}{2^n}\times\frac{|\B|}{2^n}.
\]
The result has been very influential, and was extended several times to cover more general contexts \cite{FKG,holley,AD}. However, for the most part, these results tend to focus on distributive lattices (such as $\power$) and it is natural to wonder whether correlation persists outside of this setting.\np

In this note we aim to explore analogues of the Harris--Kleitman inequality for sets of permutations.  There are two particularly natural notions for what it means for a family of permutations to be an up-set here, and the level of correlation that can be guaranteed in these settings turns out to differ greatly.\np

We write $S_n$ for the set of all permutations of $[n]$, which throughout the paper we regard as ordered $n$-tuples of distinct elements of $[n]$. That is, if $\a \in S_n$ then $\a =(a_1,\ldots, a_n)$ where $\{a_k\}_{k\in [n]} = [n]$. Given $\a \in S_n$ and $i\in [n]$, we write $\pos (\a, i)$ for the \emph{position} of $i$ in $\a $, i.e. $\pos (\a , i) = k$ if $a_k = i$. Given $1\leq i<j\leq n$, the pair $\{i,j\}$ is said to be an \emph{inversion} in $\a$ if $\pos (\a , i) > \pos (\a ,j)$. We will write $\inv (\a )$ for the set of all inversions in $\a$. A pair $\{i,j\} \in \inv (\a)$ is \emph{adjacent} in $\a $ if $\pos (\a ,i) = \pos (\a, j)+1$.

\begin{defn}
Given a family of permutations $\A\subset S_n$, we say that: 
\begin{enumerate}[(i)]
	\item $\A $ is a \emph{strong up-set} if given 
	$\a \in\A$, any permutation obtained 
	from $\a$ by swapping the elements in a pair 
	$\{i,j\} \in \inv (\a)$ is also in $\A$.
	\item 
	$\A $ is a \emph{weak up-set} if given 
	$\a \in\A$, any permutation obtained from $\a$ by
	swapping the elements in an adjacent pair 
	$\{i,j\} \in \inv (\a)$ 
	is also in $\A$.
\end{enumerate}
\end{defn}

We remark that both strong and weak up-sets have natural interpretations in the context of posets (see Chapter 2 of \cite{ABFB}). Given $\a , \b \in S_n$ write $\a \leq _s \b$ if $\b $ can be reached from $\a $ by repeatedly swapping inversions. We write $\a \leq _w \b$ if $\b$ can be reached from $\a$ by repeatedly swapping adjacent inversions. These relations give well-studied partial orders, known as the strong Bruhat order and the weak Bruhat order respectively. A (strong or weak) up-set is then simply a family which is closed upwards in the corresponding order\footnote{Note that this is in agreement with the usual notion of an up-set in ${\cal P}(X)$, starting from the poset $\big ({\cal P}(X), \subseteq \big )$.}.  

\begin{figure}[ht]
\centering
\includegraphics[scale=0.6]{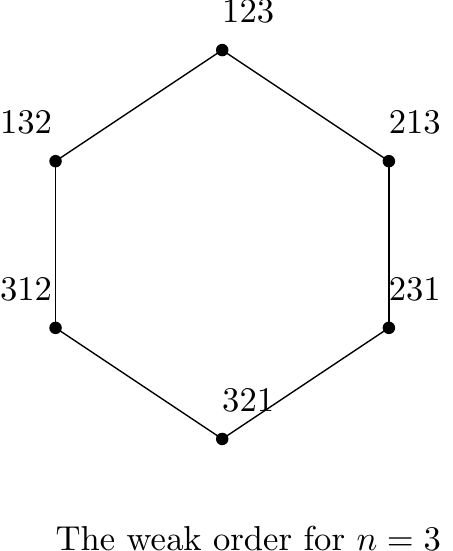}
\qquad
\includegraphics[scale=0.6]{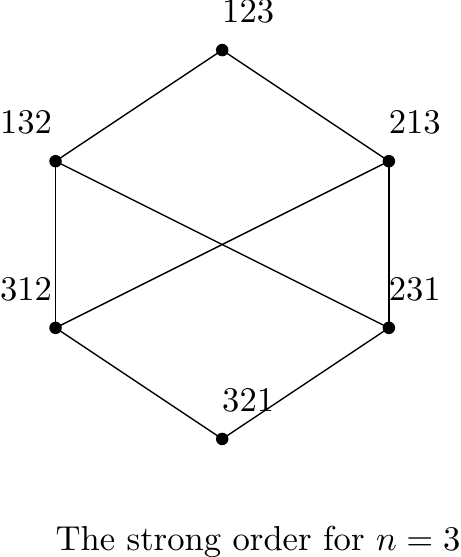}
\end{figure}

It is clear that every strong up-set is also a weak up-set, but the opposite relation is not true. For $i,j\in [n]$ let ${\cal U}_{ij} = \{\a \in S_n: i \mbox{ occurs before } j \mbox { in } \a\}$. If $i<j$, then ${\cal U}_{ij}$ is a weak up-set but not a strong up-set (except ${\cal U}_{1n}$). For example, when $n=3$, the family ${\cal U}_{12}=\{123,132,312\}$ is a weak up-set but not a strong up-set (because $213$ differs from $312$ by swapping $2$ and $3$ into increasing order but $213\not\in {\cal U}_{12}$).\np

Our first result is that strong up-sets are positively correlated in the sense of Harris--Kleitman. That is, if $\A , \B \subset S_n$ are strong up-sets then
\[
\frac{|\A\cap\B|}{n!}\geq \frac{|\A|}{n!}\times\frac{|\B |}{n!}.
\]
As we will also consider non-uniform measures, we phrase this in a more probabilistic way. We will say that a probability measure $\P$ on $S_n$ is \emph{positively associated} (for the strong order) if $\P(\A\cap\B)\geq\P(\A)\cdot \P(\B)$ for all strong up-sets $\A$ and $\B$. 

\begin{theorem}
\label{thm:strong}
The uniform measure $\P(\A)=\frac{|\A|}{n!}$ on $S_n$ is positively associated (for the strong order).
\end{theorem}

For weak up-sets the situation is more complicated. We saw that both ${\cal U}_{12}$ and ${\cal U}_{23}$ are weak up-sets and each has size $n!/2$. But ${\cal U}_{12}\cap {\cal U}_{23}$ is the family of permutations in which $1,2,3$ occur in ascending order. So, for $n \geq 3$ we have
\[
\frac{|{\cal U}_{12}\cap {\cal U}_{23}|}{n!}=\frac {1}{6}<\frac {1}{4}=\frac{|{\cal U}_{12}|}{n!}\times \frac{|{\cal U}_{23}|}{n!}.
\]
A plausible guess might be that every two up-sets ${\cal A}$ and ${\cal B}$ with size $n!/2$ achieve at least this level of correlation. Surprisingly, this turns out to be far from the truth; such an $\A $ and $\B $ can be almost disjoint.

\begin{theorem}
\label{thm:weak}
Let $0<\alpha , \beta<1$ be fixed. Then there are weak up-sets $\A, \B \subset S_n$ with $|\A| = \lfloor \alpha n!\rfloor $,  $|\B| = \lfloor \beta n!\rfloor $ and 
$|\A \cap \B | = \max (|\A | + |\B | - |S_n|,0) + o(n!)$.
\end{theorem}

The correlation given in Theorem \ref{thm:weak} is (essentially) minimal, since \emph{any} two families $\A , \B \subset S_n$ satisfy $|\A \cap \B | \geq \max \big ( |\A| + |\B | - |S_n|, 0\big )$.\np

Theorem \ref{thm:weak} shows in quite a strong sense that the uniform measure on $S_n$ is not positively associated under the weak order. Our next result will prove positive association for a wider collection of measures under the strong order, giving a different generalisation of Theorem \ref{thm:strong}.\np

Before describing these measures, we first give an alternative representation of elements of $S_n$ (essentially the Lehmer encoding of permutations -- see Chapter 11.4 of \cite{L}). Given $\a \in S_n$ we can associate a vector ${\bf f}(\a) = (f_1,\ldots ,f_n) \in G_n:= [1]\times [2] \times \cdots [n]$, with 
\[
f_j := |\{i \in [j]: \pos (\a ,i) \leq \pos (\a ,j)\}|.
\] 
In other words, $f_j$ describes where element $j$ appears in the $n$-tuple $\a$ in relation to the elements from $[j]$. This gives a bijection between $S_n$ and  $G_n$, and our positively associated measures on $S_n$ are built from this connection.

\begin{defn}
Let $X_1,\ldots, X_n$ be independent random variables, where each $X_k$ takes values in $[k]$. The \emph{independently generated measure} $\mu $ defined by $\{X_k\}_{k\in [n]}$ is the following probability measure on $S_n$: given $\a \in S_n$ we have
	\begin{equation*}
		\mu ({\bf a}) 
				:= 
		\prod _{k\in [n]} {\mathbb P}\big (X_k = {\bf f}(\a)_k \big).
	\end{equation*}
We simply say that $\mu $ is independently generated if this holds for some such collection of $\{X_k\}_{k\in [n]}$.
\end{defn}

Our second positive result applies to independently generated measures.

\begin{theorem}
\label{thm:strong-ig}
Every independently generated probability measure on $S_n$ is positively associated.
\end{theorem}

We note that the uniform measure on $S_n$ is independently generated, taking $X_k$ to simply be uniform on $[k]$. Thus Theorem \ref{thm:strong-ig} implies Theorem \ref{thm:strong}.\vspace{1mm}

We note that one special case of an independently generated measure is the Mallows measure \cite{mallows}. Recalling the definition of $\inv (\a )$ above, the Mallows measure with parameter $0<q\leq 1$ is defined by setting 
\[
\P(\a )\propto  q^{|\inv (\a )|}.
\]
That is, $\P(\a ) = \left(\sum_{\a \in S_n} q^{|\inv (\a )|}\right)^{-1} \cdot q^{|\inv(\a )|}$.\np

Our results in fact go beyond independently generated measures, and it turns out that here a key idea is a notion of up-set that sits `between' the weak and strong up-sets above. This notion, which we call `grid up-sets' (defined in Section 2), provides an environment that is suitable for FKG-like inequalities. This approach will allow us to strengthen Theorem \ref{thm:strong-ig} to apply to measures satisfying more general conditions.\np

Before closing the introduction, we note that while we have stated our results for up-sets, it is easy to obtain equivalent down-set versions of Theorems 1--3 (for example, see Chapter 19 of \cite{BB}). These follow by noting that a set $A$ in a partial order $(P,<)$ is an up-set if and only if $A^c = P \setminus A$ is a down-set. Indeed, if $\mu $ is a probability measure on $P$ and $\mu (A \cap B) \geq \mu (A) \cdot \mu (B)$ then we obtain the complementary inequality
\begin{align*}
	\mu (A^c \cap B^c) 
		= 
	1 - \mu (A) - \mu (B) + \mu (A \cap B) 
		& \geq 
	1 - \mu (A) - \mu (B) + \mu (A) \cdot 
	\mu (B)\\ 
	& = \mu (A^c) \cdot \mu (B^c).
\end{align*}

The plan of the paper is as follows. In Section 2 we prove our positive association results. Here we give a self-contained proof of Theorem \ref{thm:strong-ig}. We also introduce grid up-sets and use them to extend Theorem \ref{thm:strong-ig}. In Section 3 we prove Theorem \ref{thm:weak}, constructing weak up-sets with bad correlation properties. Section 4 gives some applications of our main results, to families of permutations defined with bounded `displacements', sequential domination properties, as well as to left-compressed set systems. Finally, in Section 5, we raise some questions and directions for further work.

\section{Correlation for strong up-sets}

In this section we will prove Theorem \ref{thm:strong-ig}. As noted in the introduction, the uniform case is an immediate corollary (Theorem \ref{thm:strong}). The proof will use induction on $n$. To relate a family of permutations of $[n]$ with a family of permutations of some smaller ground set, we `slice' according to position of element $n$. Given a family ${\cal A} \subset S_n$ and $k\in [n]$ let $\A _k \subset S_{n-1}$ denote those permutations obtained by deleting the appearance of element `$n$' from $\a \in \A $ with $\pos (\a ,n) = k$. That is: 
\begin{align*}
\A_k:= \big \{(a_1,a_2,\dots ,a_{n-1})\in S_{n-1} : (a_1,\dots , a_{k-1},n,a_k,\dots ,a_{n-1})\in\A \big \}.
\end{align*}
In the next simple lemma we collect two properties of the slice operation which will be useful later.

\begin{lemma}
\label{lem:slice}
If $\A\subset S_n$ is a strong up-set and the slices $\A_1,\A_2,\dots,\A_n\subset S_{n-1}$ are defined as above then:
\begin{enumerate}[(i)]
\item $\A_k$ is a strong up-set for all $k\in [n]$, and
\item $\A_1\subset\A_2\subset\A_3\dots\subset\A_n$.
\end{enumerate}
\end{lemma}

\begin{proof}
Part (i) is immediate. To see (ii), note that if $\a \in\A_k$ then we have $(a_1,\dots ,a_{k-1}, n, a_k,\dots ,a_{n-1})\in\A$. Now, as $\A $ is a strong up-set and $n>a_k$, the pair $\{a_k,n\} \in \inv (\a )$ and we find $(a_1,\dots ,a_k, n,a_{k+1},\dots , a_{n-1})\in\A$, giving $\a \in\A_{k+1}$.
\end{proof}

We will also need the following simple and standard arithmetic inequality, which will be used to relate the conditional probabilities of the slices in $S_{n-1}$ to probabilities in $S_n$. We provide a proof for completeness.

\begin{lemma}
\label{lem:abt}
Let $u_1,\dots,u_n,v_1,\dots,v_n,t_1,\dots,t_n\in [0,\infty )$ with $u_1\leq \ldots \leq u_n$, $v_1 \leq \ldots \leq v_n$ and $\sum _{k=1}^n t_k \leq 1$. Then
\[
\sum_{k=1}^n t_ku_kv_k\geq\left(\sum_{k=1}^n t_ku_k\right)\left(\sum_{k=1}^n t_kv_k\right).
\]
\end{lemma}


\begin{proof}
For convenience set $u_0 = v_0 = 0$. Then, for all $k\in [n]$ set $x_k = u_k-u_{k-1}$ and $y_k = v_k - v_{k-1}$. 
Note that the conditions on $u_k$ and $v_k$ give $x_k, y_k\geq 0$. Now,
\[
\sum_{k=1}^n t_ku_kv_k=\sum_{k=1}^n t_k(x_1+\dots+x_k)(y_1+\dots+y_k)=\sum_{i,j} r_{i,j}x_iy_j,
\]
where
\[
r_{i,j}=\left\{\begin{array}{cl}
t_i+\dots+t_n & \text{ if } i\geq j,\\
t_j+\dots+t_n & \text{ if } i\leq j.\\
\end{array}\right.
\]
Similarly,
\[
\left(\sum_{k=1}^n t_ku_k\right)\left(\sum_{\ell =1}^n t_{\ell }v_{\ell }\right)
	=
\left(\sum_{k=1}^n t_k \big (\sum _{i = 1}^kx_{i }\big )\right)\left(\sum_{\ell =1}^n t_{\ell }\big (\sum _{j = 1}^{\ell }y_{j}\big )\right)=\sum_{i,j} s_{i,j}x_iy_j,
\]
where $s_{i,j}=(t_i+\dots+t_n)(t_j+\dots+t_n)$.
As $t_k\geq 0$ for all $k\in [n]$ and $\sum _{k=1}^n t_k \leq 1$, we see that $r_{i,j}\geq s_{i,j}$ for all $i,j$ and the result follows.
\end{proof}

\begin{proof}
[Proof of Theorem \ref{thm:strong-ig}]
We wish to show that if $\mu $ is an independently generated probability measure on $S_n$ and ${\cal A}, {\cal B} \subset S_n$ are strong up-sets in $S_n$ then $\mu ({\cal A} \cap {\cal B}) \geq \mu ({\cal A}) \cdot \mu ({\cal B})$. We will prove this by induction on $n$. The statement is trivial for $n=1$. Assuming that the statement holds for $n-1$, we will prove it for $n$.

To begin, note that as $\mu $ is independently generated on $S_n$, it is defined by independent random variables $\{X_i\}_{i\in [n]}$. Take $\nu $ to denote the independently generated measure on $S_{n-1}$ defined by the independent random variables $\{X_i\}_{i\in [n-1]}$. By definition of $\mu $, if $\a = (a_1,\ldots, a_n) \in S_n$ with $a_k = n$ then setting $\a _{[n-1]} := (a_1,\ldots, a_{k-1}, a_{k+1},\ldots, a_n) \in S_n$ we have  
	$$\mu ( \a ) = \nu ( \a _{[n-1]} ) 
	\cdot {\mathbb P}(X_n=k).$$
It follows that given any family ${\cal F} \subset S_n$ we have  
$\mu ({\cal F}| X_n=k) = \nu ({\cal F}_k)$. Note that the measure $\nu $ does not depend on $k$, which is important below.\np

With this in hand, suppose that $\A,\B\subset S_n$ are strong up-sets. Then 
	\begin{equation*}
		\mu \big ( {\cal A} \cap {\cal B} \big ) 
			= 
		\sum _{k\in [n]} 
		{\mathbb P}(X_n = k)
		\mu \big ( {\cal A} \cap {\cal B} | X_n = k\big )
			= 
		\sum _{k\in [n]} 
		{\mathbb P}(X_n = k)
		\nu \big ( ({\cal A} \cap {\cal B})_k\big ),
	\end{equation*}
where the second equality follows by the previous paragraph. Clearly we have $({\cal A}\cap {\cal B})_k = {\cal A}_k \cap {\cal B}_k$. Moreover, as both ${\cal A}_k$ and ${\cal B}_k$ are strong up-sets by Lemma \ref{lem:slice} (i) and $\nu$ is independently generated, by induction we have $\nu ({\cal A}_k \cap {\cal B}_k) \geq \nu ({\cal A}_k) \cdot \nu ({\cal B}_k)$. Applying this above gives 
	\begin{align*}
		\mu ( {\cal A} \cap {\cal B} ) 
			&\geq 
		\sum _{k\in [n]} 
		{\mathbb P}(X_n = k) \cdot  
		\nu ( {\cal A}_k ) \cdot 
		\nu ( {\cal B}_k ) 
			= 
		\sum _{k\in [n]} t_k u_k v_k, 
	\end{align*}
where $t_k = {\mathbb P}(X_n = k)$, $u_k = \nu ( {\cal A}_k )$ and  $v_k = \nu ( {\cal B}_k )$.
Note now from Lemma \ref{lem:slice} (ii) that we have $u_1\leq \ldots \leq u_n$, $v_1\leq \ldots \leq  v_n$ and $\sum _{k\in [n]} t_k =1$. Thus the hypothesis of Lemma \ref{lem:abt} applies, and this lemma gives
\begin{align*}
	\mu ({\cal A} \cap {\cal B}) 
			\geq 
	\sum _{k\in [n]} t_k u_k v_k 
			& \geq
	\Big ( \sum _{k\in [n]} t_k u_k \Big )
	\Big ( \sum _{k\in [n]} t_k v_k \Big )\\	
			& =
		\Big ( \sum _{k\in [n]} 
		{\mathbb P}(X_n = k) 
		\nu ( {\cal A}_k )\Big ) 
		\Big ( \sum _{k\in [n]} 
		{\mathbb P}(X_n = k) 
		\nu ( {\cal B}_k ) \Big )\\
			&= 
		\mu ({\cal A} ) \cdot \mu ({\cal B} ).
\end{align*}
This completes the proof of the theorem.
\end{proof}

In contrast to this self-contained proof, our second proof will use the machinery of the FKG inequality in the following form.

\begin{theorem}[FKG inequality \cite{FKG}]
\label{fkg}
Let $L$ be a finite distributive lattice and let $\P$ be a probability measure on $L$ satisfying
\[
\P(x\wedge y) \cdot \P(x\vee y)\geq\P(x) \cdot \P(y).
\]
for all $x,y\in L$. Then any up-sets $\A,\B\subset L$ satisfy 
$\P(\A\cap\B)\geq\P(\A) \cdot \P(\B).$
\end{theorem}

To make use of Theorem \ref{fkg} recall that the permutations $S_n$ are in one-to-one correspondence with elements of the grid $G_n:= [1] \times [2]\times \cdots \times [n]$, where $\a \in S_n$ is indentified with ${\bf f}(\a ) \in G_n$. Using this correspondence we will transfer the `grid' partial order $\leq _g $ on $G_n$ to $S_n$, where ${\bf f} \leq _g {\bf g}$ for ${\bf f}, {\bf g} \in G_n$ if ${\bf f}_i \leq {\bf g}_i$ for all $i\in [n]$.

\begin{defn}
The \emph{grid order} $\leq _g$ on $S_n$ is given by defining $\a \leq _g \b $ if ${\bf f}({\bf a}) \leq _g {\bf f}({\bf b})$ when viewed as elements of $G_n$. A family $\A \subset S_n$ is a \emph{grid up-set} if whenever ${\bf a} \in\A$ and $\b \in S_n$ with $\a \leq _g \b $ then $\b \in \A $.
\end{defn}

We now use the FKG inequality to $G_n$ to give a second proof of Theorem \ref{thm:strong-ig}. In fact this approach strengthens the result in two ways: it applies to grid up-sets rather than just strong up-sets, and it applies to measures satisfying a more general FKG-type condition.\np

Let $\a , \b \in S_n$. As $G_n$ is a distributive lattice we can define $\a \vee \b $ and $\a \wedge \b $ in $S_n$ in  natural way: let $\a \vee \b , \a \wedge \b $ be the unique elements of $S_n$ with:
\begin{align*}
{\bf f}({\bf a}\vee \b )_k&=\max \big \{{\bf f}(\a)_k,{\bf f}(\b)_k \big \}; \qquad 
{\bf f}({\bf a}\wedge \b )_k =\min \big \{{\bf f}(\a)_k,{\bf f}(\b)_k \big \}.
\end{align*}

\begin{theorem}
\label{thm:strong-grid}
Suppose that $\P$ is a probability measure on $S_n$ with
\begin{equation}
	\label{equation: join meet relation}
\P(\a \vee \b )\cdot \P(\a \wedge \b) \geq \P(\a )\cdot \P(\b )
\end{equation}
for all $\a , \b \in S_n$. 
Then any grid up-sets $\A,\B \subset S_n$ satisfy 
$\P(\A\cap\B) \geq \P(\A)\cdot \P(\B)$.
\end{theorem}

\begin{proof}[Proof of Theorem  \ref{thm:strong-grid}]
Transfer $\mu $ from $S_n$ to $G_n$, by setting 
$\mu ({\bf f}(\a)) = \mu (\a)$ for all $\a \in S_n$. As 
${\bf f}: S_n \to G_n$ is a bijection this defines $\mu $ on $G_n$. By choice of the operations $\vee $ and $\wedge $ on $S_n$ above, \eqref{equation: join meet relation} implies that 
	\begin{equation*}
		\P({\bf f} \vee {\bf g} )\cdot \P({\bf f} \wedge {\bf g} ) \geq \P( {\bf f} )\cdot \P({\bf g} )
	\end{equation*}
for all ${\bf f}, {\bf g} \in G_n$. The result now follows by applying Theorem \ref{fkg} to $G_n$.
\end{proof}

To complete our second proof of Theorem \ref{thm:strong-ig}, by Theorem \ref{thm:strong-grid}, it is enough to show that (i) every strong up-set is a grid up-set and (ii) that \eqref{equation: join meet relation} holds for independently generated measures. This is content of the next two lemmas.

\begin{lemma}
\label{lem:grid-strong}
If $\a, \b \in S_n$ with $\a \leq _g \b $ then $\a \leq _s \b $. Consequently, every strong up-set in $S_n$ is also a grid up-set.
\end{lemma}

\begin{proof}
Suppose that $\a , \b \in S_n$ where $(\a ,\b )$ is a covering relation in the grid order. That is, there is $i\in [n]$ with ${\bf f}(\b )_i = {\bf f}(\a )_i + 1$ and ${\bf f}(\a )_j ={\bf f}(\b )_j$ for all $j\neq i$. It suffices to show that $\a \leq _s \b $ by transitivity, since every relation in $\leq _g$ can be expressed as a sequence of covering relations.\np

Let $\pos (\a , i) = k$ and take $\ell  > k$ minimal so that $\pos (\a , j) = \ell$ for some $j < i$; such a choice of $\ell $ must exist since $(\a , \b )$ is a covering relation with ${\bf f}(\b )_i = {\bf f}(\a )_i+1$. It is clear that $\a $ and $\b $ differ only in position $k$ and $\ell $, where $a_k = b_{\ell } = i$ and $a_{\ell } = b_k = j$. Thus $\{i,j \} \in \inv (\a)$ and swapping these entries we obtain $\b $, i.e.  $\a \leq _s \b $.\np

Lastly, if $\A $ is a strong up-set with $\a \in \A $ and $\a \leq _g \b $ then $\a \leq _s \b $, and so $\b \in \A $. Thus $\A $ is a grid up-set, as required.
\end{proof}

\begin{lemma}
\label{lem:ig}
Inequality \eqref{equation: join meet relation} holds for every independently generated probability measure $\P$ on $S_n$.
\end{lemma}

\begin{proof}
Suppose that $\P $ is an independently generated probability measure on $S_n$, defined by the independent random variables $\{X_k\}_{k\in [n]}$. Then for every $\a \in S_n$ we have 
\[
\P(\a )= \prod _{k\in [n]} {\mathbb P}\big (X_k = {\bf f}(\a) _k \big ).
\]
Then given $\a , \b \in S_n$ and $\a \vee \b $ and $\a \wedge \b $ as above, we have 
\begin{align*}
\P(\a \vee \b ) \cdot \P(\a \vee \b ) 
	& = 
\prod _{k\in [n]} \Big [ {\mathbb P}\big (X_k = \max ({\bf f}(\a) _k, {\bf f}(\b) _k) \big )\\ 
	& \quad \qquad \qquad \times {\mathbb P}\big (X_k = \min ({\bf f}(\a) _k, {\bf f}(\b) _k) \big ) \Big ]\\ 
	& = 
\prod _{k\in [n]} \Big ( {\mathbb P}\big (X_k = {\bf f}(\a) _k \big ) \cdot {\mathbb P}\big (X_k = {\bf f}(\b) _k \big ) \Big ) = \mu ({\bf a}) \cdot \mu ({\bf b}).
\end{align*}
Thus \eqref{equation: join meet relation} holds with equality for all $\a , \b \in S_n$, as required.
\end{proof}

Above we defined the grid order on $S_n$ in such a way that it was isomorphic to the usual product ordering on $G_n$. Analysing the proof of Lemma \ref{lem:grid-strong} more carefully gives an alternative description of the grid order on $S_n$ in terms of certain switches. Given $1 \leq i < j \leq n$, recall that $\{i,j\}$ is an inversion in $\a$  if $\pos (\a ,j) = k < \ell =\pos (\a ,i)$. We will say that $\{i,j\}$ is a \emph{dominated inversion} in $\a $ if additionally $a_m \geq i,j$ for all $m\in [k,\ell ]$.
Then $\a \leq _g \b $ if $\b$ can be reached from $\a $ by a sequence of operations, each consisting of swapping the elements from a dominated inversion.\vspace{1mm}

\begin{figure}[ht]
\centering
\includegraphics[scale=1]{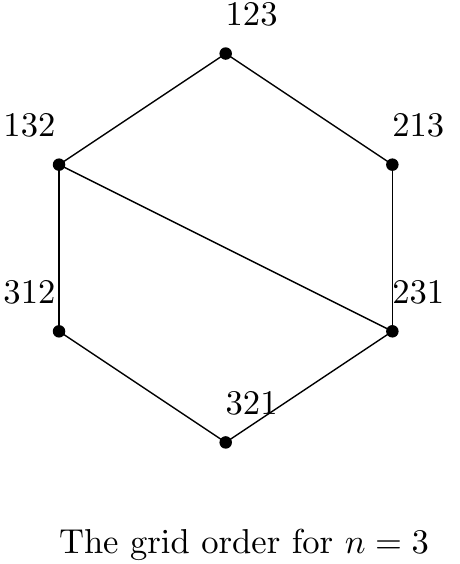}
\end{figure}

\section{No correlation for weak up-sets}

In this section we construct weak up-sets which are very far from being positively correlated. We will need the following simple concentration result.

\begin{lemma}
\label{lem:concentration}
Let $0<\gamma , \delta , \varepsilon <1$. Let $U, V \subset [n]$ with $|U| = \gamma n$ and $|V| = \delta n$. Select $\a \in S_n$ uniformly at random and consider the random variable $N(\a) := |\{ i\in U: a_i \in V\}|$. Then $\mathbb{P} \big ( N>(\gamma +\varepsilon)|V| \big ) 
\rightarrow 0$ as $n\rightarrow\infty$.
\end{lemma}

\begin{proof}
For each $i\in [n]$ let ${1}_i: S_n \to \{0,1\}$ denote the Bernoulli random variable with ${1}_i({\bf a}) = 1$ iff $a_i \in V$. Then ${\mathbb E} [ {1}_i ] = |V|/n$ for all $i\in [n]$. Noting that $N = \sum _{i\in U} {1}_i$, linearity of expectation gives ${\mathbb E}[N] = \gamma |V|$.\np

To calculate the variance of $N$, note that 
 ${\mathbb E}[1_i\cdot 1_j] \leq |V|^2/n^2$ for $i\neq j$. Since $N = \sum _{i\in U} 1_i$, this gives 
$${\mathbb E}[N^2] = \sum _{i\in U} {\mathbb E}[1_i^2] 
+ \sum _{i\neq j \in U} {\mathbb E}[1_i 1_j] \leq \gamma |V| + \gamma ^2 |V|^2,$$ 
and so ${\mathbb V}\mbox{ar}(N) = {\mathbb E}[N^2] - \big ( {\mathbb E}[N] \big )^2 \leq \gamma |V|$. Chebyshev's inequality then gives 
${\mathbb P}(N> (\gamma + \varepsilon )|V|) \leq {\mathbb P}(|N - {\mathbb E}[N]| \geq \varepsilon |V|) \leq \gamma / \big ( \varepsilon ^2 |V| \big ) \to 0$ as $n\to \infty $.
\end{proof}

We are now ready for the proof of Theorem \ref{thm:weak}.

\begin{proof}[Proof of Theorem \ref{thm:weak}]

Given $0<\alpha , \beta, \varepsilon <1$, we require to find weak up-sets ${\cal A}, {\cal B} \subset S_n$ for large $n$, which satisfy $|{\cal A}| \geq \alpha n!$, $|{\cal B}| \geq \beta n!$ and $|{\cal A}\cap {\cal B}| \leq ( \max (\alpha + \beta - 1 , 0) + 5\varepsilon ) n!$. Indeed, by deleting minimal elements from such ${\cal A}$ and ${\cal B}$ we obtain weak up-sets of size $\lfloor \alpha n!\rfloor $ and $\lfloor \beta n! \rfloor $ as in the theorem.\np

To begin, set $m=\lceil ( \frac{\alpha }{\alpha + \beta } )n\rceil$ so that $\frac{m}{n-m}=\frac{\alpha }{\beta }+o(1)$. Consider the function $g: S_n \to [m]$ where $g({\bf a})$  equals the number of elements from $[m]$ which do not appear after element $m$ in $\bf a$. That is,
	$$g({\bf a}) 
		:=
	\big |\big \{i \in [m]: \pos ({\bf a}, i ) \leq 
	\pos ({\bf a}, m) \big \}\big |.$$
Noting that $g$ is non-decreasing under switching inversions, we see that ${\cal A} := \{ {\bf a} \in S_n: g({\bf a}) \geq (1 - \alpha )m \}$ is a weak up-set in $S_n$. Also noting that the families $L_i = \{{\bf a} \in S_n: g({\bf a}) = i\}$ for $i\in [m]$ partition $S_n$ into equal-sized sets, we obtain $|{\cal A}| = \sum _{i\in [(1-\alpha )m,m]} |L_i| \geq \alpha n!$.\np

Our second family ${\cal B}$ is defined similarly. Let $h : S_n \to [n-m+1]$, where $h({\bf a})$ equals the number of elements from $[m, n]:=\{m, m+1,\ldots, n\}$ which do not appear before element $m$ in ${\bf a}$. That is, 
	$$h({\bf a}) 
		:=
	\big |\big \{i \in [m,n]: 
	\pos ({\bf a}, i ) \geq 
	\pos ({\bf a}, m) \big \}\big |.$$
Reasoning as above, we find ${\cal B} := \{{\bf a} \in S_n: h({\bf a}) \geq (1 -\beta ) (n-m+1)\}$ is a weak up-set and $|{\cal B}| \geq \beta n!$.\np

Having defined both families, it only remains to upper bound $|{\cal A} \cap {\cal B}|$. Here it is helpful to consider two further families. 
\begin{itemize}
\item For ${\bf a} \in S_n$ let $N_1({\bf a}) := \big | \big \{ k\in U_1: a_k \in V_1 \big \} \big |$, where $U_1 = [(1-\alpha -\varepsilon )n]$ and $V_1 = [m]$. Then ${\cal E}_1 := \big \{{\bf a} \in S_n: N_1({\bf a}) \geq (1-\alpha )|V_1| \big \}.$

\item For ${\bf a} \in S_n$ let $N_2({\bf a}) := \big | \big \{ k\in U_2: a_k \in V_2 \big \} \big |$, where $U_2 = [(\beta + \varepsilon )n, n]$ and $V_2 = [m,n]$. Then ${\cal E}_2 := \big \{{\bf a} \in S_n: N_2({\bf a}) \geq (1-\beta )|V_2| \big \}.$ 

\end{itemize}
The functions $N_1$ and $ N_2$ are defined as in Lemma \ref{lem:concentration}, and so we have $|{\cal E}_1|, |{\cal E}_2| \leq \varepsilon n!$, provided $n \geq n_0(\alpha , \beta , \varepsilon )$.\np

We claim that every ${\bf a} \in {\cal C} := ({\cal A} \cap {\cal B}) \setminus ({\cal E}_1 \cup {\cal E}_2)$ satisfies 
\begin{equation}
\label{equation: interval bound} 
\pos ({\bf a},m) \in I : = \big [ (1-\alpha - \varepsilon )n, (\beta + \varepsilon )n \big ].
\end{equation}
Note that this will complete the proof of the theorem, since it gives 
$$|{\cal A} \cap {\cal B}| \leq |{\cal C}| + |{\cal E}_1| + |{\cal E}_2| \leq \Big ( \frac {|I|+1}{n} \Big  ) n! + 2 \varepsilon n! \leq \big ( \max ( \alpha + \beta - 1, 0 ) + 5\varepsilon \big ) n!.$$	
To prove the claim, take $\a \in {\cal A} \cap {\cal B}$. Note that if $\pos (\a , m) <(1-\alpha -\varepsilon)n$ then ${\bf a} \in {\cal E}_1$ since 
\begin{align*}
N_1(\a) = |\{k \in U_1: a_k \in V_1\}| & = |\{i \in [m]: \pos (\a , i) \leq (1-\alpha -\varepsilon )n\}|\\ &\geq g(\a) \geq (1-\alpha )m = (1-\alpha )|V_1|,
\end{align*} 
The first equality is by definition of $N_1$, the second equality holds by double counting, the first inequality follows from by definition of $g$ and the fact that $\pos (\a , m) \leq (1-\alpha -\epsilon )n$, and the final inequality holds as $\a \in \A $.\np

Similarly, if $\pos (\a , m) > (\beta + \varepsilon )n$ then $\a \in {\cal E}_2$, since
\begin{align*}
N_2(\a) 
	= 
|\{k \in U_2: a_k \in V_2\}| & = |\{i \in [m,n]: \pos (\a , i) \geq (\beta +\varepsilon )n\}|\\ 
	&\geq 
h(\a) \geq (1 - \beta )(n-m+1) = (1-\beta )|V_2|.
\end{align*} 
Again, the first two equalities hold by definition of $N_2$ and by double counting respectively. The first inequality follows from the definition of $h$ and the fact that $\pos (\a , m) > (\beta + \varepsilon )n$, and the final inequality holds as $\a \in \B $.\np

We have shown that if $\a \in (\A \cap \B) \setminus ({\cal E}_1 \cup {\cal E}_2) = {\cal C}$ then $\a$ satisfies \eqref{equation: interval bound} which, as described above, completes the proof.
\end{proof}

\section{Examples and an application}

Several natural families of permutations enjoy the property of being strong up-sets. In the first subsection we present a number of examples of these. Together these provide a wide variety of families for which positive correlation results can be deduced from Theorem \ref{thm:strong} and Theorem \ref{thm:strong-ig}. For instance, we shall see that for a random permutation ${\bf a} \in S_n$ (chosen uniformly or following an independently generated measure), the event that no element is displaced by more than a fixed distance $d$ by 
${\bf a}$ and the event that ${\bf a}$ contains at most $k$ inversions are positively correlated. Likewise, each of these events is positively correlated with the event that at least $u$ elements from $\{1,\dots,v\}$ occur among the first $w$ positions in ${\bf a}$.

In the second subsection we will give an application of Theorem \ref{thm:strong} to the correlation of left-compressed set families.

\subsection{Examples of strong up-sets}

\subsubsection*{Layers} 

For each $k\in \big [ 0, \binom {n}{2} \big ]$ let ${\cal L}_k := \{ \a \in S_n : |\inv (\a)| = k\}$. Then it is easily seen that the family ${\cal L}_{\geq k} := \cup _{i\geq k} {\cal L}_i$ is a strong up-set. In words, this is the set of all permutations which can be written as a product of at most $\binom{n}{2}-k$ adjacent transpositions. 

\subsubsection*{Band-like permutations} 
Our next example is based on considering how much each element is moved by a permutation.
Given a permutation $\a \in S_n$ and an element $i\in[n]$, the \emph{displacement} of $i$ in $\a $ is given by $\disp (\a , i) := |i - \pos (\a ,i)|$. We will say $\a $ is a \emph{$t$-band permutation} if $\disp (\a ,i) \leq t$ for all $1\leq i\leq n$. 

\begin{lemma}
\label{lem:band}
The $t$-band permutations in $S_n$ form a strong up-set.
\end{lemma}

\begin{proof}
Suppose that $\a \in S_n$ is a $t$-band permutation and that $\{i,j\} \in \inv (\a )$. Let $\b $ be the permutation obtained from $\a $ by swapping $i$ and $j$. It is clear that $\disp (\a , k)=\disp (\b , k)$ for all $k\notin \{i,j\}$. A simple case check also gives
\begin{enumerate}[(a)]
\item $\disp (\b ,i) + \disp (\b , j) \leq \disp (\a ,i) + \disp (\a , j)$, and

\item $|\disp (\b ,i)-\disp (\b ,j)|\leq |\disp (\a ,i)-\disp (\a ,j)|$.
\end{enumerate}
As $\a$ is a $t$-band permutation we have $\disp (\a ,i), \disp (\a ,j) \leq t$ and so it follows that $\disp (\b ,i), \disp (\b ,j) \leq t$, i.e. $\b $ is also a $t$-band permutation. 
\end{proof}

In fact this argument shows rather more. Given $\a \in S_n$, the \emph{displacement list} ${\bf d}(\a )$ is the vector given by:
\[
{\bf d}(\a ) := \big (\disp (\a ,1),\dots,\disp (\a ,n) \big ).
\]
Now, given a set of vectors $\mathcal{D}\subset\{0,1,\dots,n-1\}^n$, we can form the family of permutations $\A({\cal D}) : =\{\a \in S_n : {\bf d}(\a ) \in\mathcal{D}\} \subset S_n$. That is, those permutations in $S_n$ whose displacement lists lie in $\mathcal{D}$.

\begin{defn}
A set of permutations $\A$ is said to be \emph{band-like} if $
\A = \A ({\cal D})$ for some set $\mathcal{D}\subset\{0,1,\dots,n-1\}^n$ which is closed under:
\begin{itemize}
\item reordering the entries,
\item decreasing any entry,
\item replacing two entries of an element of ${\cal D}$ with new entries so that neither the sum or difference of these entries increases. 
\end{itemize}
\end{defn}

The argument of Lemma \ref{lem:band} shows that:

\begin{lemma}
\label{lem:bandlike}
Any band-like set of permutations in $S_n$ is a strong up-set.
\end{lemma}

In addition to $t$-band permutations, examples of band-like sets include $\{ \a \in S_n : \sum_{i=1}^n \disp (\a , i)\leq t\}$ and $\{ \a \in S_n : \sum_{i=1}^n \disp (\a , i)^2\leq t\}$. 

\subsubsection*{Sequentially dominating permutations}

Our final example arises from assigning weight and thresholds as follows.
Given a sequence of real weights ${\bf w}=(w_1,\dots,w_n)$ with $w_1\geq w_2\geq\dots\geq w_n$ and thresholds ${\bf t}=(t_1,\dots, t_n)$ we consider the family 
\[
{\cal D}({\bf w},{\bf t}) := \Big \{\a = (a_1,\ldots, a_n) \in S_n : \sum_{i=1}^m w_{a_i}\geq t_m \text{ for all }m \Big \}.
\]
Since the weights are decreasing, such families are closed under swapping inversions and so form strong up-sets.\np 

Some common families arise in this way, including the families of permutations which satisfy `at least $a$ elements from $\{1,\dots,b\}$ occur among the first $c$ positions'. Indeed, such families can be written as ${\cal D}({\bf w},{\bf t})$, where
\[
{\bf w}=(\underbrace{1,\dots,1}_{b},\underbrace{0,\dots,0}_{n-b}),
\]
with $t_i =a $ if $i = c$, and $t_i =0$ otherwise.\bigskip

Many specific examples follow from these general families. For instance, 

\begin{cor}
Let ${\bf a}$ be a random permutation chosen under an independently generated probability measure on $S_n$. Then, for any $k,l,m,u,v,w\in\mathbb{N}$, any two of the following events are positively correlated:
\begin{itemize}
\item There are at most $k$ inversions in ${\bf a}$,
\item No element is displaced by more than $l$ by ${\bf a}$,
\item The sum over all elements of the displacements in ${\bf a}$ is at most $m$,
\item The first $w$ positions of ${\bf a}$ contain at least $u$ of the elements $\{1,\dots,v\}$.
\end{itemize}
\end{cor}
\qed \vspace{2mm}

Amusingly, the families of permutations constructed in the proof of Theorem \ref{thm:weak} (our non-correlation result for weak up-sets) can be described using weights in a superficially similar way to a sequentially dominated family. Given a non-increasing sequence of weights and any thresholds, we may define the set of all permutations satisfying that the sum of all entries up to and including element $m$ is at least $t_m$ for all $m$. More precisely,
\[
{\cal D'}({\bf w},{\bf t}) := \Big \{\a = (a_1,\ldots, a_n) \in S_n : \sum_{i=1}^m w_{a_i}\geq t_{a_m} \text{ for all }m \Big\}. 
\]
In general this is not an up-set in the strong or weak sense. However, if we take weights $u_1=u_2=\dots u_k=1$, $u_{k+1}=\dots=u_n=0$ with threshold $s_k=k/2$ and weights $v_1=v_2=\dots v_k=0$, $v_{k+1}=\dots=v_n=-1$ with threshold $t_k=-k/2$ then the two families ${\cal D'}({\bf u},{\bf s})$ and ${\cal D'}({\bf v},{\bf t})$ are precisely those constructed in the proof of Theorem \ref{thm:weak}.

\subsection{Maximal chains and left-compressed up-sets}

A family of sets $\A\subset\power$ is \emph{left-compressed} if for any $1\leq i<j\leq n$, whenever $A\in\A$ with $i\not\in A,j\in A$ we also have $(A\setminus\{j\})\cup\{i\}\in\A$. See \cite{BB} for background and a number of useful applications of compressions. It is not hard to show that if $\A$ and $\B$ are left-compressed $r$-uniform families (that is, each consists of $r$-element subsets of $[n]$) then they are positively correlated in the sense that
\[
\frac{|\A\cap\B|}{\binom{n}{r}}\geq \frac{|\A|}{\binom{n}{r}}\times\frac{|\B|}{\binom{n}{r}}.
\]
However, in general left-compressed families may not be positively correlated; indeed, they may simply be disjoint if the families have different sizes. Below we use Theorem \ref{thm:strong} to give a natural measure of the similarity of non-uniform families from which positive correlation for left-compressed families follows.\np
 
A maximal chain in $\power$ is a nested sequence of sets $C_0\subset C_1\subset\dots\subset C_n$ with $C_i\subset X$ and $|C_i|=i$. A permutation $\a $ of $X$ can be thought of as a maximal chain in $\power$ by identifying $\a $ with the family of sets forming initial segments from $\a$; that is setting $C_i:=\{a_1,\dots,a_i\}$ for all $i\in [0,n]$.\np

If $\A$ is a family of sets, we write $c(\A)$ for the number of maximal chains which contain an element of $\A$. Note that if $\A$ is $r$-uniform then the probability that a uniformly random maximal chain meets $\A$ is proportional to $|\A|$ and so in this case ${c(\A)}/{n!}={|\A|}/{\binom{n}{r}}$. If $\A$ and $\B$ are families of sets then we write $c(\A,\B)$ for the number of maximal chains that meet both $\A$ and $\B$. We will use $c(\A)$ as our measure of the size of $\A$ and $c(\A,\B)$ as our measure of the intersection (or similarity) of $\A$ and $\B$. With this notion, the following Theorem can be interpreted as saying that left-compressed families are positively correlated.
 
\begin{theorem}
\label{thm:chains}
If $\A$ and $\B$ are left-compressed families from ${\cal P}(X)$ then
\[
\frac{c(\A,\B)}{n!}\geq \frac{c(\A)}{n!}\times\frac{c(\B)}{n!}.
\]
\end{theorem}

\begin{proof}
Let $C(\A)$ denote the set of all permutations of $X$ which correspond to chains meeting $\A$ and $C(\B)$ be the set of all permutations of $X$ which correspond to chains meeting $\B$. Since $\A$ and $\B$ are left-compressed $C(\A)$ and $C(\B)$ are strong up-sets in $S_n$. Applying Theorem \ref{thm:strong} gives the result.
\end{proof}

We remark that while the functional $c({\cal A})/n!$ is thought of as a measure $\A $, it is not a probability measure on ${\cal P}(X)$ since additivity fails (e.g. consider the partition ${\cal P}(X) = \cup _i \binom {X}{i}$).\np

A number of further variations on this result are possible (e.g. if $\A $ is left-compressed and $\B $ is right-compressed then $\A$ and $\B$ are negatively correlated). For example, given a family ${\cal F} \subset {\cal P}(X)$ and a maximal chain ${\cal C}$, let $N_{\cal F}({\cal C}):= |{\cal C} \cap {\cal F}|$. Identifying permutations $\a \in S_n$ with maximal chains as above, we obtain the following.

\begin{theorem}
\label{thm:rv-chains}
Let $\A, \B \subset {\cal P}(X)$ be left-compressed families and suppose that ${\cal C}$ is a maximal chain from ${\cal P}(X)$ chosen uniformly at random. Then for any $k,l$ we have
\[
\mathbb{P}\big (N_{\cal A}({\cal C})\geq k, N_{\B}({\cal C})\geq l\big )
	\geq
\mathbb{P}\big (N_{\A}({\cal C})\geq k\big )
	\cdot 
\mathbb{P}\big (N_{\B}({\cal C})\geq l\big ).
\] 
\end{theorem}

\section{Open questions}

One general question is to determine which other measures on $S_n$ satisfy positive association. A particularly appealing class of measures to consider here are those given by a 1-dimensional spatial model. Spatial models of this kind are much studied in statistical physics. See \cite{BR,BU} for examples of such results.\np 

Let $x(1),x(2),\dots, x(n)\in\mathbb{R}$ with $x(1)\leq x(2)\leq\dots\leq x(n)$. We will regard these as $n$ particles placed in increasing order on the real line. A permutation $\a = (a_1,\ldots, a_n) \in S_n$ gives rise to a permutation of these particles. In this model, any point $x(i)$ is displaced by $|x(i)-x(\pos (\a , i))|$. The total displacement is $\sum_{i} |x(i) - x(\pos (\a , i))|$. We define the associated measure on $S_n$ by
\[
\P(\a )\propto q^{\sum_i |x(i)-x(\pos (\a ,i))|}. 
\]
More generally, given a non-decreasing function $V:\mathbb{R}^+\rightarrow\mathbb{R}^+$, define
\[
\P(\a )\propto  q^{\sum_{i} V(x(i)-x(\pos (\a ,i)))}. 
\]
These definitions are special cases of the well-studied Boltzmann measures in which points are picked in $\mathbb{R}^d$ and more general functions in the exponent of $q$ are allowed.\np 

We suspect that all measures defined in this way have positive association. However we do not have a proof of this, even in special cases. The following three cases all seem interesting.
\begin{q}[Equally spaced points]
Is the measure $\P$ defined by
\[
\P(\a )\propto  q^{\sum_{i} |i - \pos (\a , i)|} 
\]
positively associated?
\end{q}

This corresponds to taking $x(i)=i$ and $V(u) = |u|$.

\begin{q}[Middle gap]
Let $m(\a )=|\{ k : 1\leq k\leq n/2, n/2<a_k\leq n\}|$ be the number of elements which are moved `across the middle gap' by $\a $. Is the measure $\P$ defined by
\[
\P(\a )\propto  q^{m(\a )}
\]
positively associated?
\end{q}

This corresponds to taking $x(i) = 0$ for $i\in [\frac {n}{2}]$ and $x(i) = 1$ if $i\in [\frac {n}{2}+1,n]$ and $V(u) = 1$ if $u <0$ and $V(u) = 0$ otherwise. 

\begin{q}[Fixed points]
Let $f(\a )=|\{ k : a_k=k\}|$ be the number of fixed points of $\a $. Is the measure $\P$ defined by
\[
\P(\a )\propto  q^{n-f(\a )} 
\]
positively associated?
\end{q}

This corresponds to taking any distinct $\{x(i)\}_{i\in [n]}$ and setting $V(0) = 0$ and $V(u) = 1$ otherwise.\np

Lastly, the correlation behaviour seen in the strong and weak orders are extreme, with the first displaying  Harris--Kleitman type correlation (Theorem \ref{thm:strong}) and the second displaying worst possible correlation (Theorem \ref{thm:weak}). It seems interesting to understand how correlation behaviour emerges between these extremes. 

\begin{defn}
Given $t \in [n]$, a family of permutations $\A\subset S_n$ is a $t$\emph{-up-set} if 
given $\a \in\A$, any permutation obtained from $\a$ by swapping the elements in a pair $\{i,j\} \in \inv (\a)$  with $|\pos (\a, i) - \pos (\a ,j)| \leq t$ is also in $\A$.
\end{defn}

Note that if $t = 1$ then a $t$-up-set is simply a weak up-set. On the other hand, for $t = n$ then a $t$-up-set is a strong up-set. Thus we can think of $t$-up-sets as interpolating between the weak and strong notions as we increase $t \in [n]$. It seems natural to investigate the correlation behaviour of $t$-up-sets.

\begin{q}[Correlation for $t$-up-sets]
Given $\alpha > 0$, does there exist 
$\beta  >0$ such that the following holds: given $n\in {\mathbb N}$ and $t = \lceil \alpha n \rceil $, any two $t$-up-sets ${\cal A} , {\cal B} \subset S_n$ with $|{\cal A}|, |{\cal B}| \geq \alpha n!$ satisfy $|{\cal A} \cap {\cal B}| \geq \beta n!$.
\end{q}

\noindent {\bf Acknowledgement.} We are grateful to Ander Holroyd for bringing the Mallows measures to our attention. We would also like to thank the referees for their careful reading and helpful comments.


\begin{thebibliography}{}
\bibitem{AD} R. Ahlswede and D.E. Daykin, {An inequality for the weights of two families of sets, their unions and intersections}, \emph{Probab. Theory Relat. Fields}, \textbf{43}(3) (1978), 183--185.
\bibitem{BU} V. Betz and D. Ueltschi, {Spatial random permutations and infinite cycles}, \emph{Comm. Math.
Phys.} \textbf{285} (2009), no. 2, 469--501.
\bibitem{BR} M. Biskup and T. Richthammer, {Gibbs measures on permutations over one-dimensional discrete point sets}, \emph{Ann. Appl. Probab.} \textbf{25}, Number 2 (2015), 898--929.
\bibitem{ABFB} A. Bj\"orner and F. Brenti, \emph{Combinatorics of Coxeter Groups}, \textbf{231}, Springer, New York, 2005.
\bibitem{BB} B. Bollob\'as, \emph{Combinatorics}, Cambridge Univ. Press, London/New York, 1986.
\bibitem{FKG} C.M. Fortuin, P.W. Kasteleyn and J. Ginibre, {Correlation inequalities on some partially ordered sets}, \emph{Comm. Math. Phys.} \textbf{22}(2) (1971), 89--103.
\bibitem{harris} T.E. Harris, {A lower bound for the critical probability in a certain percolation}, \emph{Math. Proc. Cam. Phil. Soc.}, \textbf{56} (1960), 13--20.
\bibitem{holley} R. Holley, {Remarks on the FKG inequalities}, \emph{Commun. Math. Phys.}, \textbf{36}(3), (1974), 227--231.
\bibitem{kleitman} D.J. Kleitman, {Families of non-disjoint subsets}, \textit{J. Combin. Theory} \textbf{1} (1966), 153--155.
\bibitem{L} M. Lothaire, \emph{Algebraic combinatorics on words}, \textbf{90}, Cambridge University Press, 2002.
\bibitem{mallows} C.L. Mallows, {Non-null ranking models. I}, \emph{Biometrika}, \textbf{44} (1957), 114--130.


\end{thebibliography}
\end{document}